\documentclass[a4paper,11pt,reqno]{amsart}

\usepackage{setspace}
\let\oldtocsection=\tocsection
\let\oldtocsubsection=\tocsubsection
\renewcommand{\tocsection}[2]{\hspace{-1mm}\bf\oldtocsection{ #1}{#2}}
\renewcommand{\tocsubsection}[2]{\hspace{6.9mm}\oldtocsubsection{#1}{#2}}

\usepackage[normalem]{ulem}
\usepackage{color}
\usepackage{amsmath}
\usepackage{amssymb}
\usepackage{array}
\usepackage{amsthm}
\usepackage[hidelinks]{hyperref}
\usepackage{xcolor}
\usepackage{enumitem}
\usepackage{tikz}
\usepackage{mathrsfs}
\usepackage{xfp}

\usepackage{float}
\theoremstyle{plain}
\newtheorem{thm}{Theorem}[section]
\newtheorem{prop}[thm]{Proposition}
\newtheorem{cor}[thm]{Corollary}
\newtheorem{lem}[thm]{Lemma}

\theoremstyle{definition}
\newtheorem{defn}[thm]{Definition}

\theoremstyle{remark}
\newtheorem{rem}[thm]{Remark}

\newcommand{\m}{\mathfrak{m}}
\newcommand{\p}{\partial}

\newcommand{\scal}{\operatorname{scal}}
\newcommand{\ric}{\operatorname{Ric}}
\newcommand{\trace}{\operatorname{tr}}
\newcommand{\divergence}{\operatorname{div}}

\newcommand{\dv}{\, dV}

\newcommand{\hg}{\widehat{g}}

\newcommand{\rgp}{(\mathring g,\mathring \pi)}
\newcommand{\onab}{\mathring\nabla}

\newcounter{mnotecount}[section]

\numberwithin{equation}{section}

\newtheorem{introthm}{\bf Theorem}

\title[The volume-renormalized mass from a Hamiltonian perspective]
{The volume-renormalized mass from a Hamiltonian perspective} 
\author[M. Dahl]{Mattias Dahl}
\address{Institutionen f\"or Matematik \\
Kungliga Tekniska H\"ogskolan \\
100 44 Stockholm \\
Sweden} 
\email{dahl@kth.se}

\author[K. Kr\"{o}ncke]{Klaus Kr\"{o}ncke}
\address{Institutionen f\"or Matematik \\
Kungliga Tekniska H\"ogskolan \\
100 44 Stockholm \\
Sweden} 
\email{kroncke@kth.se}

\author[S. McCormick]{Stephen McCormick}
\address{Institutionen f\"or teknikvetenskap och matematik \\
Lule{\aa} tekniska universitet \\
971 87 Lule\aa \\
Sweden} 
\email{stephen.mccormick@ltu.se}

\begin{document}
%%%%%%%%%%%%%%%%%%%%%%%%%%%%%%%%%%%%%%%%%%%%%%%%%%%%%%%%%%%%%%%%%%%%%%%%%%%%%%%%

\begin{abstract}
We demonstrate that the volume-renormalized mass for asymptotically hyperbolic manifolds recently introduced by the authors \cite{DKM2023arxiv} can be deduced from a reduced Hamiltonian perspective. In order to do this, we first use Michel's formalism \cite{Mic11} of mass invariants to extend the definition of the volume-renormalized mass to initial data sets.

We consider spacetimes that are foliated by asymptotically Poincar\'e--Einstein Riemannian manifolds in the spirit of the Milne model of cosmology and reduce the ADM Hamiltonian to an unconstrained Hamiltonian system, analogous to the work of Fischer and Moncrief for spatially compact spacetimes \cite{FischerMoncrief2002}. We find that the reduced Hamiltonian in this case recovers the volume-renormalized mass. We then analyze the first and second variation of the reduced Hamiltonian and demonstrate that it is non-increasing over the evolution and constant only for self-similar spacetimes.
\end{abstract}

\maketitle

%%%%%%%%%%%%%%%%%%%%%%%%%%%%%%%%%%%%%%%%%%%%%%%%%%%%%%%%%%%%%%%%%%%%%%%%%%%%%%%%
\section{Introduction}
%%%%%%%%%%%%%%%%%%%%%%%%%%%%%%%%%%%%%%%%%%%%%%%%%%%%%%%%%%%%%%%%%%%%%%%%%%%%%%%%

The volume-renormalized mass is a geometric invariant of asymptotically hyperbolic manifolds, introduced by the authors in \cite{DKM2023arxiv}. It is essentially a combination of an ADM-like boundary integral and the renormalized volume of the manifold such that the dominant terms cancel out, yielding a finite quantity under weaker decay assumptions than usually required for each to individually be well-defined. A precise definition is given in Section \ref{Sec-prelim}. Although we were able to demonstrate that the volume-renormalized mass satisfies properties one may expect of a mass invariant, our initial motivation for the definition was purely geometric, as a regularisation of the total scalar curvature. In this article we motivate the volume-renormalized mass from a more physical perspective, complementing the geometric motivation previously given. In particular, we demonstrate that the volume-renormalized mass is equal to the value of a reduced Hamiltonian for a particular class of spacetimes related to the Milne model of cosmology. Additionally, this work motivates us to extend the definition of the volume-renormalized mass to initial data sets $(M,g,\pi)$. To this end, we employ Michel's formalism for mass invariants \cite{Mic11} to obtain a generalised quantity that depends additionally on $\pi$ and is well-defined under appropriate decay assumptions.

The Milne model is a cosmological model, simply taken to be a light cone of Minkowski space parametrized so that each slice of constant time is isometric to a rescaled hyperboloid (see Figure \ref{fig-milne} in Section \ref{sec-hamredmilne}). We consider a class of spacetimes that we call \textit{asymptotically Milne-like} in the sense that they are foliated by asymptotically Poincar\'e--Einstein (APE) manifolds.

We follow the programme of Fischer and Moncrief \cite{FischerMoncrief1997,FischerMoncrief2002} to perform a Hamiltonian reduction of the Einstein equations, which removes the gauge freedom and gives an unconstrained Hamiltonian system. This is in contrast to the standard ADM Hamiltonian, which is a constrained Hamiltonian system. Fischer and Moncrief developed this programme in the study of self-similar solutions to the Einstein equations with compact Cauchy surfaces, under the assumption that the spacetime can be foliated by closed constant mean curvature (CMC) hypersurfaces. We carry out this Hamiltonian reduction for asymptotically Milne-like spacetimes that are foliated by noncompact spacelike CMC hypersurfaces. The reduced Hamiltonian that we obtain is the volume-renormalized mass, up to a rescaling by a cosmological time parameter.

Heuristically, the reduced Hamiltonian is derived as follows. We work with the assumption that the region of spacetime under consideration can be foliated by asymptotically hyperbolic CMC Riemannian hypersurfaces. Via the resolution of the Yamabe problem in this setting, the conformal method then allows us to parametrise the set of solutions to the Einstein constraint equations for each (constant) mean curvature in terms of a constant scalar curvature metric $g$ and transverse-traceless tensor density $p$. Treating $g$ and (a time rescaling of) $p$ as the canonical variables leads to an unconstrained Hamiltonian system. Note that in order to fully reduce the Hamiltonian, so that each point in the phase space corresponds to a unique vacuum solution, one must additionally fix a gauge condition.

After establishing that the value of reduced Hamiltonian equals the volume-renormalized mass, we study its variational properties. In our previous article \cite{DKM2023arxiv} it was shown that critical points of the volume-renormalized mass over a space of constant scalar curvature metrics are Einstein metrics (see also \cite{McCormick2025}). In the same spirit, we consider the unconstrained problem and show that critical points of the reduced Hamiltonian on the reduced phase space also coincide with Einstein metrics (with vanishing reduced momentum). Furthermore, compute the second variation to study the minimization properties. Finally, we demonstrate that the volume-renormalized mass is indeed non-increasing along Einstein evolution, and constant if and only iff the spacetime is Milne-like.

We state simplified versions of our main results presented here.
\begin{introthm}[Theorem \ref{thm-well-defined}]
	The volume-renormalized mass of an APE initial data set $(g,\pi)$ given by
		\[ \begin{split}
			\m_{VR,\mathring g}(g,\pi)
			= 
			\lim_{R \to \infty}
			\Big(& 
			\m_{{\rm ADM},\mathring{g}}(g,R)+2(n-1)RV_{\mathring{g}}(g,R) \\
			& +2\int_{B_R} \left( \trace_g\overline\pi+n(n-1) \right) dV_{\mathring g} \Big)
		\end{split} \] 
		is well-defined and finite if the Hamiltonian constraint is integrable.
\end{introthm}
\begin{introthm}\label{intro-ham}
	The volume-renormalized mass is equal to a reduced Hamiltonian \`a la Fischer--Moncrief for asymptotically Milne-like spacetimes foliated by spacelike CMC hypersurfaces, up to a time rescaling.
\end{introthm}
Note that Theorem \ref{intro-ham} is not stated as a theorem in the main body of the text, but rather is developed throughout Sections \ref{Sec-CompactHam} \ref{sec-hamredmilne}.
\begin{introthm}[Theorem \ref{thm-firstvar}]
	Critical points of the volume-renormalized mass over the reduced phase space are exactly at Einstein metrics with vanishing (reduced) momentum.
\end{introthm}
\begin{introthm}[Corollary \ref{cor-secondvar}]
	If the Einstein operator at a critical point has positive first eigenvalue, then that critical point is a minimum of the volume-renormalized mass.
\end{introthm}
\begin{introthm}[Theorem \ref{thm-mono}]\label{intro-mono}
	The volume-renormalized mass is non-increasing along the Einstein evolution, and constant if and only if the spacetime is Milne-like.
\end{introthm}

The article is structured as follows. First, Section \ref{Sec-prelim} sets some basic definitions and terminology. Then Section \ref{Sec-IDVR} derives an appropriate definition of the volume-renormalized mass of an initial data set. Section \ref{Sec-CompactHam} develops the basic ideas of the reduced Hamiltonian for compact initial data with boundary before Section \ref{sec-hamredmilne} introduces the reduced Hamiltonian for asymptotically Milne-like spacetimes. Section \ref{Sec-variations} studies the first and second variation of the volume-renormalized mass. Then finally in Section \ref{Sec-mono} we calculate the derivative of the volume-renormalized mass under Einstein evolution to establish Theorem \ref{intro-mono}.

\section*{Acknowledgements}
We appreciate support from the G\"oran Gustafsson Foundation (KK) and GS Magnusons Fond (SM).

%%%%%%%%%%%%%%%%%%%%%%%%%%%%%%%%%%%%%%%%%%%%%%%%%%%%%%%%%%%%%%%%%%%%%%%%%%%%%%%%
\section{Preliminaries}\label{Sec-prelim}
%%%%%%%%%%%%%%%%%%%%%%%%%%%%%%%%%%%%%%%%%%%%%%%%%%%%%%%%%%%%%%%%%%%%%%%%%%%%%%%%

\subsection{APE manifolds and their volume-renormalized mass}

In this section we recall some definitions including the definition of the volume-renormalized mass from \cite{DKM2023arxiv}, as well as establish notation. 

\begin{defn}
Let $\overline M$ be a compact manifold with compact boundary $\partial \overline M$. Let $\rho : \overline M \to [0,\infty)$ be a smooth {\em boundary defining function}, which means that $\rho^{-1}(0) = \partial \overline M$ and $d\rho|_{\partial \overline M} \neq 0$. Let $M = \overline M \setminus \partial \overline M$. We say that a Riemannian metric $g$ on $M$ is {\em conformally compact} of class $C^{k,\alpha}$, if there is a $C^{k,\alpha}$-Riemannian metric $b$ on $\overline M$ so that $g = \rho^{-2} b$. In this case, the sectional curvatures of $g$ tend to $-|d\rho|^2_{h}$ at $\partial \overline M$. If $|d\rho|^2_{h} = 1$ so that all sectional curvatures tend to $-1$ at $\partial \overline M$ we say that $(M,g)$ is {\em asymptotically hyperbolic}, or simply AH. The Riemannian manifold $(\overline M,b)$ is called the {\em conformal background} of $(M,g)$ and $(\p \overline M,[b|_{\partial \overline M}])$ is called the {\em conformal boundary} of $(M,g)$.
\end{defn}

We will make use of the radial function $r$ defined by $\rho =e^{-r}$, which we use to define the weighted H\"older spaces $C^{k,\alpha}_\delta(M)=e^{-\delta r}C^{k,\alpha}(M)$, equipped with the norm
\[
\|u\|_{k,\alpha,\delta} = \|e^{\delta r}u\|_{C^{k,\alpha}(M)}.
\]
Here, $\delta\in\mathbb R$ and $C^{k,\alpha}(M)$ denotes the standard H\"older space with the norm $\|\cdot\|_{C^{k,\alpha}(M)}$. Weighted H\"older spaces of sections of bundles are defined in the usual way (see, for example, \cite{Lee06}).

Let $\mathcal{R}(M)$ be the space of Riemannian metrics on $M$, that is the space of positive definite sections of the bundle $S^2 T^*M$ of symmetric bilinear forms on $TM$. For a fixed AH manifold $(\mathring{M},\mathring{g})$ with $\mathring{g} \in C^{k,\alpha}(S^2 T^*M)$ we define the space of Riemannian metrics on $\mathring{M}$ asymptotic to $\mathring{g}$ as
\[
\mathcal{R}^{k,\alpha}_{\delta}(\mathring{M},\mathring{g})
= \left\{
g \in \mathcal{R}(M) \mid
g-\mathring{g}\in C^{k,\alpha}_{\delta}(S^2 T^*\mathring{M})
\right\}.
\]

\begin{defn}\label{defn-AH}
Let $(M,g)$ and $(\mathring{M} ,\mathring{g})$ be AH manifolds with conformal backgrounds $\overline M$ and $\mathring{\overline M}$, respectively. We say that $(M,g)$ is \emph{asymptotic to} $(\mathring{M} ,\mathring{g})$ of order $\delta>0$ with respect to $\varphi$, if there are bounded and closed sets $K\subset M$, $\mathring{K}\subset \mathring{M}$ and a $C^{k+1,\alpha}$-diffeomorphism $\varphi: \overline M \setminus K\to \mathring{\overline M}\setminus \mathring{K}$
of manifolds with boundary such that 
$\varphi_*g \in \mathcal{R}^{k,\alpha}_{\delta}(\mathring{M} \setminus \mathring{K}, \mathring{g})$.
\end{defn}

With $\mathring{g}, K,\mathring{K}$ and $\varphi$ as in Definition \ref{defn-AH}, we define
\[
\mathcal{R}^{k,\alpha}_{\delta}(M,\varphi,\mathring{g})
=
\left\{ 
g \in \mathcal{R}(M) \mid 
g\in C^{k,\alpha}(S^2 T^*M), \varphi_*g \in \mathcal{R}^{k,\alpha}_{\delta}(\mathring{M} \setminus \mathring{K}, \mathring{g})
\right\}.
\]

\begin{defn}\label{defn:APE_2} 
We say an AH manifold $(M,g)$  of class $C^{k,\alpha}$, $k\geq 2$ is \emph{asymptotically Poincar\'{e}--Einstein} (APE) of order $\delta$ if $|\ric_g+(n-1)g|_g\in C^{k-2,\alpha}_\delta(M)$ for some  $\frac{n-1}{2}<\delta<n-1$ satisfying $\delta \leq k+\alpha$. The set of all such metrics is denoted by $\mathcal{R}^{k,\alpha}_{\delta}(M)$. We will always impose these restrictions on $\delta,k$ and $\alpha$.
\end{defn}

\begin{rem}
All metrics in $\mathcal R^{k,\alpha}_\delta$ are asymptotic to each other in the sense of Definition \ref{defn-AH} (Prop. 2.6 of \cite{DKM2023arxiv}).
\end{rem} 

Given an AH manifold $(M,g)$ asymptotic to $(\mathring{M} ,\mathring{g})$ with diffeomorphism $\varphi$, we choose the boundary defining functions on $\overline M$ and $\mathring{\overline M}$ so that $\rho = \mathring{\rho} \circ \varphi$ and $r = \mathring{r} \circ \varphi$ on $N\setminus K$. Define the sets
\[
B_R = \{x\in {M} \mid {r}(x)<R \} \subset M, \qquad
\p B_R = \{x\in {M} \mid {r}(x)=R \} \subset{M},
\]
and
\[
\mathring{B}_R=\{x\in \mathring{M} \mid \mathring{r}(x)<R\}\subset\mathring{M},\qquad
\p\mathring{B}_R=\{x\in \mathring{M} \mid \mathring{r}(x)=R\}\subset\mathring{M}.
\]
For $R$ so large that $\varphi(\p B_R)=\p\mathring{B}_R$ for $R$, let
\begin{align*}
\m^\varphi_{{\rm ADM},\mathring{g}}(g,R)
&= \int_{\p\mathring{B}_R}(\divergence_{\mathring{g}}(\varphi_*g)-d\trace_{\mathring{g}}(\varphi_*g))(\nu_{\mathring{g}})\dv_{\mathring{g}},\\
RV^\varphi_{\mathring{g}}(g,R)
&= \int_{B_R}\dv_g-\int_{\mathring{B}_{R}}\dv_{\mathring{g}},
\end{align*}
where $\nu_{\mathring{g}}$ is the outward unit normal to $\p \mathring B_R$ in $(\mathring M,\mathring{g})$.

\begin{defn} \label{defn-VRmass}
Let $(M,\mathring{g})$ be asymptotically hyperbolic. 	The (Riemannian) {\em volume-renormalized mass} $\mathfrak{m}_{\rm{VR},\mathring{g}}(g)$ of $g$ with respect to $\mathring{g}$ is given by
\[
\mathfrak{m}_{\rm{VR},\mathring{g}}(g)=
\lim_{R \to \infty}
\big( \m_{{\rm ADM},\mathring{g}}(g,R)+2(n-1)RV_{\mathring{g}}(g,R) \big),
\]
where we omit reference to $\varphi$ as the definition is independent of $\varphi$ under rather mild assumptions, see \cite[Thm~3.14]{DKM2023arxiv}.
\end{defn}

It is also shown in \cite{DKM2023arxiv} that $\m_{VR,\mathring g}(g)$ is well-defined and finite provided that both $(\scal_{\mathring g}+n(n-1))$ and $(\scal_{g}+n(n-1))$ are integrable. In Section \ref{Sec-IDVR}, we generalise this Riemannian definition to the initial data setting.

\subsection{Hamiltonians and the phase space}

It is well-known that the Einstein equations can be recast as a Hamiltonian system. The standard ADM approach to this is to choose a time function $t$ and then write an $(n+1)$-dimensional Lorentzian metric $h$ in terms of a lapse function $N$, a shift vector $X$, and $n$-dimensional metrics $g$ induced on the hypersurfaces $\Sigma$ of constant $t$, as
\begin{equation}\label{eq-ADMsplit}
h = -N^2dt^2+g_{ij}(dx^i+X^idt)\otimes(dx^j+X^jdt).
\end{equation}
The Einstein equations is then be decomposed into a system of evolution equations and constraint equations on the hypersurface. Taking a Riemannian metric $g$ on the hypersurface as the canonical position variable for the Hamiltonian system, recall that the conjugate momentum $\pi$ to $g$ is the symmetric $2$-tensor density given in terms of the second fundamental form $K$ of $\Sigma$ by
\[
\pi^{ij} = \left(\trace_g(K)g^{ij}-g^{ik}g^{jl}K_{kl}\right)\dv_g.
\]
\begin{rem}
	Note that we follow the sign convention for $K$ used by Fischer and Moncrief \cite{FischerMoncrief2002}. 
\end{rem}

For the sake of presentation we will also make use of the de-densitized momentum $\overline\pi$ defined by $\pi=\overline\pi \dv_g$ throughout this article. The vacuum constraint equations are then given by $\Phi(g,\pi)=0$ where the constraint map $\Phi$ is defined by
\begin{equation}\label{eq-constraintmap}
\begin{split}
\Phi_0(g,\pi)
&=
\scal_g\,\dv_g+\left(\frac 1{n-1}(\trace_g(\overline\pi))^2-|\overline\pi|_g^2\right)\dv_g, \\
\Phi_i(g,\pi)
&= 
2g_{ik}\nabla_j\pi^{jk}.
\end{split}
\end{equation}
For a closed manifold $\Sigma$ the Hamiltonian is given by \cite{Gourgoulhon}
\[
\mathcal{H} = 
-\int_\Sigma \left(N\Phi(g,\pi)+X^i\Phi_i(g,\pi)\right),
\]
while for noncompact manifolds or manifolds with boundary an additional boundary term appears in $\mathcal H$. For example, for asymptotically flat manifolds this boundary term is proportional to the ADM mass.

The phase space for the Einstein equations is then a space $\mathcal P_0$ of pairs $(g,\pi)$, which can be viewed as a cotangent bundle over a manifold of Riemannian metrics on $\Sigma$. Given a point $(g_0,\pi_0)\in\mathcal P_0$ satisfying $\Phi(g_0,\pi_0)=0$ then the Hamiltonian flow in $\mathcal P_0$ starting from $(g_0,\pi_0)$ generates a curve $(g_t,\pi_t)\in\mathcal P_0$ satisfying $\Phi(g_t,\pi_t)=0$, corresponding to the Einstein evolution equations.

In our setting, we are interested in a class of initial data where $g$ is APE and $\pi$ has asymptotics motivated by the Milne foliation of Minkowski space. Let $\mathcal{K}(M)$ be the space of sections of the bundle $S^2 TM \otimes \Lambda^n M$ of symmetric contravariant $2$-tensor densities and let  
\begin{align*}
\mathcal K^{k-1,\alpha}_\delta(M)
&=
\Big\{ \pi\in \mathcal{K}(M) \mid \\
&\varphi_*\pi +(n-1)\mathring g^{-1}dV_{\mathring g}\in C^{k-1,\alpha}_\delta(S^2 T(\mathring M\setminus \mathring K) \otimes \Lambda^n (\mathring M\setminus \mathring K))
\Big\},
\end{align*}
where $\varphi$ is the diffeomorphism given by Definition \ref{defn-AH}.

We define the APE phase space as 
\[
\mathcal P^{k,\alpha}_{\delta}(M)
= \mathcal R^{k,\alpha}_\delta(M) \times \mathcal K^{k-1,\alpha}_\delta(M).
\]
Note that the definition of $\mathcal K^{k-1,\alpha}_\delta$ is independent of the choice of $\mathring g\in \mathcal R^{k,\alpha}_{\delta}$. 

\begin{defn}\label{def-APEdata}
We say an initial data set $(M,g,\pi)$ is APE if $(g,\pi)\in\mathcal P^{k,\alpha}_\delta(M)$.
\end{defn}

%%%%%%%%%%%%%%%%%%%%%%%%%%%%%%%%%%%%%%%%%%%%%%%%%%%%%%%%%%%%%%%%%%%%%%%%%%%%%%%%
\section{Initial Data Definition of the Volume-Renormalized Mass} 
\label{Sec-IDVR}
%%%%%%%%%%%%%%%%%%%%%%%%%%%%%%%%%%%%%%%%%%%%%%%%%%%%%%%%%%%%%%%%%%%%%%%%%%%%%%%%

In \cite{DKM2023arxiv}, we defined the volume-renormalized mass for APE manifolds and considered the scalar curvature to be bounded below by $-n(n-1)$. Such manifolds can be viewed as initial data $(M,g,\pi)$ for the Einstein equations satisfying the dominant energy condition with $\pi^{ij}=-(n-1)g^{ij}dV_g$. In particular, for this choice of $\pi$ the constraint map \eqref{eq-constraintmap} reduces to
\[ \begin{split}
\Phi_0(g,\pi)&=(\scal_g+n(n-1))\dv_g,\\
\Phi_i(g,\pi)&=0.
\end{split} \]
We view the definition given in \cite{DKM2023arxiv} as a Riemannian definition, and in this section we develop an initial data definition of the volume-renormalized mass.

In particular, we approach the volume-renormalized mass from the perspective of Michel's formulation of mass-like invariants \cite{Mic11}, and give a definition of volume-renormalised mass for APE initial data sets. We fix background initial data to be a given APE manifold $(M,\mathring g)$ equipped with conjugate momentum $\mathring\pi=-(n-1)\mathring g^{ij}dV_{\mathring g}$, and such that $\scal_{\mathring g}+n(n-1)$ is integrable. Throughout this section we consider initial data $(g,\pi)\in\mathcal P^{k,\alpha}_\delta$, and we use the convention that all indices are raised and lowered with the metric $\mathring g$.

Let $V=(N,X)$ consist of a bounded function $N$ and a bounded vector field $X$ on $M$. This implies that $\langle V,\Phi\rgp\rangle = N\Phi(g,\pi)+X^i\Phi_i(g,\pi)$ is integrable.

Linearizing the constraints about $\rgp$ gives
\begin{equation}\label{eq-linconstraintwithQ}
D\Phi{\rgp}(h,q) = \Phi(g,\pi) - \Phi\rgp + Q(h,q),
\end{equation}
where $(h,q)=(g-\mathring g,\pi-\mathring \pi)\in T_{(\mathring g,\mathring \pi)} \mathcal P^{k,\alpha}_\delta(M)$ is tangent to the APE phase space,
and where $Q(h,q)$ denotes a term which is quadratic in $(h,q)$. In particular, $Q(h,q)$ is in $L^1(M)$ for $(g,\pi)\in\mathcal P^{k,\alpha}_\delta$ with $\delta>\frac{n-1}{2}$. From \eqref{eq-linconstraintwithQ}, we see that if $\langle V,\Phi(g,\pi)\rangle\in L^1$ and $V \in L^\infty$ then $\langle V,D\Phi{\rgp}(h,q)\rangle\in L^1$ for $(h,q)\in T_{\rgp} \mathcal P^{k,\alpha}_\delta$. Furthermore, if $V$ is not bounded then stronger decay assumptions are required on the initial data to ensure that the contribution from $\langle V, Q \rangle$ is integrable.

We next have 
\begin{equation}\label{eq-Michel-relation}
\langle V,D\Phi{\rgp}(h,q)\rangle
= \divergence_{\mathring g}\mathbb{U}(V,h,q) 
+ \langle D\Phi{\rgp}^*(V),(h,q)\rangle,
\end{equation}
where
\[ \begin{split}
\mathbb{U}^i(V,h,q)
&= 
N\left(
(\mathring\nabla_jh^{ij} - \mathring\nabla^i\trace_{\mathring g}h)
-h^{ij}\mathring\nabla_jN
-\trace_{\mathring g}h\nabla^iN)
\right)dV_{\mathring g} \\
&\qquad
+ 2X^jq_j^i
+ 2X^j(n-1)h_{j}^i dV_{\mathring g}
- X^i(n-1)\trace_{\mathring{g}}(h)dV_{\mathring g}.
\end{split} \]
and the adjoint $D\Phi{\rgp}^*$ is given below. Integrating $\divergence\mathbb U$ over $M$ gives the flux integral of $\mathbb U$ at infinity. By choosing $V$ in the kernel of $D\Phi{\rgp}^*$ and $(h,q)=(g-\mathring g,\pi-\mathring \pi)$, this surface integral gives a well-defined mass-like invariant for $(g,\pi)$, provided that decay conditions are chosen such that the quadratic term $Q(h,q)$ is integrable. In particular, the asymptotics of elements of the kernel of $D\Phi{\rgp}^*$ govern the decay rates required for $(g,\pi)$ to ensure that $\divergence_{\mathring g}\mathbb{U}(V,h,q)$ is integrable, and thus obtain a mass-like invariant as a surface integral at infinity. 

Interestingly though, the combination on the right-hand side of \eqref{eq-Michel-relation} is integrable for any choice of bounded $V$, provided that the decay conditions on $g$ and $\pi$ ensure $Q(h,q)\in L^1$. That is, if $|g-\mathring g|_{\mathring g}^2$ and $|\pi-\mathring \pi|_{\mathring g}^2$ are in $L^1$ and $V$ is in $L^\infty$, then from \eqref{eq-linconstraintwithQ} and \eqref{eq-Michel-relation} we have that $(\divergence_{\mathring g}\mathbb{U}(V,h,q) + \langle D\Phi{\rgp}^*(V),(h,q)\rangle) \in L^1$ whenever $(\Phi(g,\pi) - \Phi\rgp) \in L^1$. We now show that from this observation one can recover the volume-renormalised mass.

The adjoint of the linearized constraint map at the reference data $\rgp$ is given by a standard computation, see for example \cite[Sec~IV~B]{Mic11},
\begin{equation} \label{eq-DPhistar}
\begin{split}
&\langle D\Phi_0{\rgp}^*(N),(h,q)\rangle \\
&\qquad = 
h_{ij}\left( \onab^i\onab^jN-\mathring g^{ij}\mathring \Delta N \right)dV_{\mathring g} \\
&\qquad\qquad
-N\trace_{\mathring g}h(n-1)(n-3)dV_{\mathring g}-2N\trace_{\mathring g}q , \\
&\langle D\Phi_i{\rgp}^*(X^i),(h,q)\rangle \\
&\qquad = 
(n-1)\trace_{\mathring g}h\onab_k X^k dV_{\mathring g}-2\mathcal L_X\mathring g_{ij}((n-1)h^{ij}dV_{\mathring g}+q^{ij}) \\
&\qquad =
\mathcal L_X\mathring g_{ij}
((n-1)(2h^{ij}-\frac12\trace_{\mathring g}h\,\mathring g^{ij})dV_{\mathring g}-2q^{ij}).
\end{split}
\end{equation}

In order to derive the volume-renormalized mass, we make the choice $N=1$ and $X=0$ in \eqref{eq-Michel-relation}. In the context of the Milne model, this corresponds to the unit normal to the hyperboloidal initial data slice, which is the velocity vector field for a family of comoving observers. From \eqref{eq-DPhistar}, we have
\begin{equation}\label{eq-N-1-Michel}
\langle D\Phi_0{\rgp}^*(1),(h,q)\rangle 
= -\trace_{\mathring g}h(n-1)(n-3)dV_{\mathring g} 
- 2\trace_{\mathring g}q.
\end{equation}
Consider the map $L(g,\pi)$ defined by 
\[
L(g,\pi) = \trace_g\pi\dv_g^{-1},
\]
where $\dv_g^{-1}$ is a scalar density of weight $-1$ that ``de-densitizes" $\pi$, and its linearization given by
\[
DL(g,\pi)(h,q)
= h_{ij}\pi^{ij}\dv_g^{-1} + \trace_gq\dv_g^{-1} - \frac12 \trace_g h \trace_g \pi\dv_g^{-1}.
\] 
By Taylor expanding this map, we can write
\[ \begin{split}
\trace_g\overline\pi dV_{\mathring g}
&=
\trace_{\mathring g}\mathring\pi 
+ h_{ij}\mathring\pi^{ij}
+ \mathring g_{ij}q^{ij}
- \frac12\trace_{\mathring g}h\trace_{\mathring g}\mathring\pi 
+ Q(h,q) \\
&=-
n(n-1)dV_{\mathring g} 
-(n-1)(1-\frac n2)\trace_{\mathring g}hdV_{\mathring g}
+\trace_{\mathring g}q 
+ Q(h,q)
\end{split} \]
where again $(h,q)=(g-\mathring g,\pi-\mathring \pi)$ and $Q(h,q)$ denotes a term which is quadratic in $(h,q)$ and thus is integrable, which may change from line to line. Further, $\overline\pi=\pi\dv_g^{-1}$ defines the de-densitized momentum $\pi$. Since
\[ 
dV_g = dV_{\mathring g} + \frac12 \trace_{\mathring g}h dV_{\mathring g} + Q(h,q),	
\]
we can then write
\[
\trace_{\mathring g}q
=
\left( \trace_g\overline\pi+n(n-1) \right)\dv_{\mathring g}-(n-1)(n-2)\left(\dv_g-\dv_{\mathring g}\right)+Q(h,q).
\] 
Inserting this into \eqref{eq-N-1-Michel} gives
\[ \begin{split}
\langle D\Phi_0{\rgp}^*(1),(h,q)\rangle
&=
- 2(n-1)(n-3)\left(\dv_g-\dv_{\mathring g}\right) \\
&\qquad
- 2(\trace_g\overline\pi+n(n-1))dV_{\mathring g} \\
&\qquad 
+2(n-1)(n-2)\left(\dv_g-\dv_{\mathring g}\right) dV_{\mathring g} + Q(h,q) \\
&= 
2(n-1)\left(\dv_g-\dv_{\mathring g}\right) \\
&\qquad
+2(\trace_g\overline\pi-n(n-1))dV_{\mathring g} + Q(h,q)
\end{split} \]
This leads to the following definition of the volume-renormalized mass of an initial data set.
\begin{defn}
	The volume-renormalized mass of an APE initial data set $(M,g,\pi)$ with respect to $(M,\mathring g, \mathring \pi=-(n-1)\mathring g^{-1}\dv_{\mathring g})$ is given by
\[ \begin{split}
	\m_{VR,\mathring g}(g,\pi)
	= 
	\lim_{R \to \infty}
	\Big(& 
	\m_{{\rm ADM},\mathring{g}}(g,R)+2(n-1)RV_{\mathring{g}}(g,R) \\
	& +2\int_{B_R} \left( \trace_g\overline\pi+n(n-1) \right)  dV_{\mathring g} \Big)
\end{split} \] 
\end{defn}
Furthermore, the argument above has proven:
\begin{thm} \label{thm-well-defined}
	Let $(M,g,\pi)\in\mathcal P^{k,\alpha}_{\delta}$ be an APE initial data set with $k\geq2$ and $\delta>\frac{n-1}{2}$, asymptotic to $(\mathring M, \mathring g, \mathring\pi=(n-1)dV_{\mathring g})$. Then if $\Phi(g,\pi)\in L^1$, $\m_{VR,\mathring g}(g)$ is well-defined and finite.
\end{thm}
Alternatively, writing in terms of the mean curvatures $k=\trace_g(K)$ and $\mathring k=\trace_{\mathring g}\mathring K$ we have that
\[
\lim_{R \to \infty}
\left( 
\m_{{\rm ADM},\mathring{g}}(g,R)
+2(n-1)\left( RV_{\mathring{g}}(g,R)
+\int_{B_R} \left( k-\mathring k \right)   \,dV_{\mathring g} \right)
\right)
\] 
is well-defined and finite. In the Riemannian case, the expression $k-\mathring k$ is identically zero so the quantity here reduces to the Riemannian volume-renormalized mass. 

\begin{rem}\label{rem-compatibility}
If $(\trace_g\overline\pi+n(n-1))\in L^1$, then the condition $\Phi(g,\pi)\in L^1$ is equivalent to $(\scal_g+n(n-1))\in L^1$ for the decay conditions assumed here, which was the condition in the Riemannian case \cite[Thm~3.1]{DKM2023arxiv}. This can be seen directly by writing $h=\overline\pi+(n-1)g^{-1}$ and noticing
\begin{equation*} 
\Phi(g,\pi)=\scal_g+n(n-1)+\frac1{n-1}(\trace_gh)^2-|h|_g^2-2\trace_gh.
\end{equation*}
\end{rem}

\begin{rem}
All choices of lapse and shift $(N,X)\in L^\infty$ will give a well-defined quantity, so one could then ask for natural vectors fields $X$ on $(M,\mathring g)$ to use with $N=0$ to define a kind of momentum quantity. In light of \eqref{eq-DPhistar}, a natural choice for $X$ is to take it to be a conformal Killing field for $(M,\mathring g)$. However, the conformal Killing fields for hyperbolic space are growing at infinity and therefore by the discussion above, we would require stronger decay conditions to obtain a finite momentum quantity from this argument. For this reason we focus only on the volume-renormalized mass quantity obtained above from the choice $N=1$ and vanishing $X$. 
\end{rem}

%%%%%%%%%%%%%%%%%%%%%%%%%%%%%%%%%%%%%%%%%%%%%%%%%%%%%%%%%%%%%%%%%%%%%%%%%%%%%%%%
\section{Hamiltonian Reduction on a Compact Manifold with Boundary}\label{Sec-CompactHam}
%%%%%%%%%%%%%%%%%%%%%%%%%%%%%%%%%%%%%%%%%%%%%%%%%%%%%%%%%%%%%%%%%%%%%%%%%%%%%%%%

In this section we discuss the reduced phase space and reduced Hamiltonian for a compact manifold with boundary.

\subsection{The reduced phase space}

The usual Hamiltonian for general relativity is a constrained Hamiltonian system, and solutions to Einstein's equations are given by curves in the constraint submanifold. Here we construct a reduced phase space $\mathcal P_{red}$ and a corresponding reduced (unconstrained) Hamiltonian, where each point in $\mathcal P_{red}$ generates a curve of solutions to the Einstein equations through the Hamiltonian flow.

Let $\Sigma$ be a manifold with boundary equipped with a background metric $\mathring{g}$.
We define the Yamabe invariant $\mathcal{Y}_g$ of a metric $g$ on $\Sigma$ as in \cite[Section~2]{HolstTsogtgerel2013},
\[
\mathcal{Y}_g = \inf \left\{ E(\phi) \mid \| \phi \|_{L^{\frac{2n}{n-2}}(M,g)} = 1 \right\}
\] 
where 
\[
E(\phi) 
= \int_{\Sigma} \left( |d\phi|_g^2 + \frac{n-2}{4(n-1)} \scal_g \phi^2 \right) \dv_{g}
+ \frac{n-2}{2} \int_{\partial\Sigma} H_g \phi^2 \dv_{g}
\]
and $H_g$ denotes the mean curvature of $\partial \Sigma$ with respect to the outward pointing normal.

\begin{defn}
The full phase space is defined by 
\[
\mathcal{P}(\Sigma) 
= \left\{ 
(g,\pi) \in C^{\infty}(S^2_+T^*\Sigma) \times C^{\infty}(S^2 T\Sigma \otimes\Lambda^n\Sigma) 
\mid 
g|_{\partial \Sigma}=\mathring{g}|_{\partial \Sigma}
\right\},
\]
where $S^2_+ T^*M$ is the bundle of positive definite symmetric bilinear forms on $TM$. The CMC phase space is 
\[
\mathcal{P}_{CMC}(\Sigma)
=
\left\{ (g,\pi)\in\mathcal P(\Sigma) \mid
\mathcal{Y}_g < 0, \trace_g(\overline\pi)=-n(n-1), \Phi(g,\pi)=0 \right\}.
\]
\end{defn}
The CMC phase space is the submanifold of the full phase space consisting of CMC vacuum initial data for the Einstein equations with mean curvature $\tau=-n$ and negative Yamabe invariant. 

Vacuum solutions of the Einstein equations with CMC foliations generate curves in $\mathcal P_{CMC}(\Sigma)$ as follows. Throughout we use the hat $\widehat\cdot$ to denote `physical' quantities, as opposed to rescaled versions of them used to parametrise the phase space. Let $(\Sigma\times[t_1,t_2],h)$ be a spacetime satisfying the Einstein vacuum equations, which we write via the ADM decomposition \eqref{eq-ADMsplit} as
\begin{equation} \label{eq:ADM_decomp}
\begin{split}
h &= 
-\widehat N^2dt^2 + \hg_{ij}(dx^i+\widehat X^idt)\otimes(dx^j+\widehat X^jdt) \\
&= 
-N^2dt^2 + t^2g_{ij}(dx^i+t^{-1}X^idt)\otimes(dx^j+t^{-1}X^jdt),
\end{split}
\end{equation}
where we introduced the rescaled quantities $g$ and $X$ defined by $\hg = t^2 g$ and $\widehat X = t X$ on the slice of constant $t$. We do not rescale $\widehat N$ but we write $\widehat N = N$.

For the data $(g,N,X)$ we assume the boundary conditions
\[
g=\mathring{g}, \qquad N=1, \qquad X = \mathring X,\qquad \divergence_g(X) = 0,
\]
on $\partial \Sigma$, where $\mathring X$ is some fixed vector field defined in a neighbourhood of $\partial \Sigma$.

We assume that the slicing is chosen so that each slice of constant $t$ corresponds to a leaf of the CMC foliation. We also assume that the metric $\hg$ has negative Yamabe invariant $\mathcal{Y}_{\hg} < 0$, which ensures that $\tau$ cannot vanish. The assumption on the foliation gives a relationship between $t$ and the mean curvature $\tau$ of the leaf $\Sigma_t$. We have
\begin{equation}\label{eq-evo}
\p_t\hg_{ij}=-2N\widehat K_{ij}+\mathcal L_{t^{-1}X}\hg_{ij},
\end{equation}
where $\widehat K$ is the second fundamental form of $(\Sigma_t,\widehat g)$ in the spacetime. The trace of this equation gives
\[
\trace_{\hg}\p_t\hg = -2N\tau+2\divergence_{\hg}({t^{-1}X}).
\]
Using the fact that $\hg=t^2g$, and the boundary conditions on $\p\Sigma$ we obtain 
$2tn = -2\tau$, or equivalently the relationship
\[
\tau = -n/t
\]
on $\p\Sigma$. Since $\tau$ is constant, this relationship is valid on all of $\Sigma$. We will later show that $N$ is uniquely determined by the CMC condition but $X$ is still freely specifiable, apart from the boundary conditions.

We define the momentum $\widehat{\pi}$ conjugate to $\widehat{g}$ by 
\[
\widehat{\pi}^{ij} =
\left(\trace_{\widehat{g}}(\widehat K)\widehat{g}^{ij}-\widehat{g}^{ik}\widehat{g}^{jl}\widehat K_{kl}\right)\dv_{\widehat{g}}.
\]
Since we have a solution to the vacuum Einstein equations, $(\hg,\widehat{\pi})$ is a curve in $\mathcal{P}(\Sigma)$ satisfying the constraint equations $\Phi(\hg,\widehat{\pi})=0$. Define the rescaled momentum $\pi$ by $\widehat{\pi}=t^{n-3}\pi$. For the rescaled data $(g,\pi)$ we have
\[
\Phi_0(\hg,\widehat{\pi})=t^{n-2}\Phi_0(g,\pi)=0, 
\quad
\Phi_i(\hg,\widehat{\pi})=t^{n-1}\Phi_i(g,\pi)=0.
\]
Further, we have
\[
\trace_{\hg}(\widehat{\pi})=(n-1)\trace_{\hg}(\widehat K)\dv_{\hg}=-n(n-1)t^{n-1}\dv_{g}
\]
so that
\[
\trace_{g}({\pi})=-n(n-1)\dv_{g}.
\]
We conclude that the pair $(g,\pi)$ is a curve in $\mathcal P_{CMC}(\Sigma)$.

\begin{defn}
The reduced phase space is defined by
\[
{\mathcal P}_{red}(\Sigma)
=
\left\{ (\gamma,p)\in\mathcal P(\Sigma) \mid 
\scal_\gamma=-n(n-1), \operatorname{div}_{\gamma}p=0, \trace_{\gamma}p=0 
\right\} .
\]
\end{defn}

By the conformal method we can to each $(\gamma,p)\in\mathcal P_{red}(\Sigma)$ associate a unique $(g,\pi) \in \mathcal P_{CMC}(\Sigma)$, see \cite[Section~1.2]{HolstTsogtgerel2013}. This is given by 
\begin{equation} \label{eq-redvars}
g_{ij} = \varphi^{\frac{4}{n-2}}\gamma_{ij}, 
\quad 
\pi^{ij} = \varphi^{-\frac{4}{n-2}} p^{ij} 
- (n-1) \varphi^{-\frac{4}{n-2}} \gamma^{ij} \dv_g,
\end{equation}
where $\varphi$ satisfies the Lichnerowicz equation 
\begin{equation} \label{eq-lichn}
-\frac{4(n-1)}{n-2} \Delta_\gamma \varphi 
-n(n-1) \varphi
+ n(n-1) \varphi^{\frac{n+2}{n-2}}
- |p|_\gamma^2 \varphi^{-\frac{3n-2}{n-2}}
= 0
\end{equation}
with the boundary condition $\varphi=1$ on $\partial \Sigma$. 

From \cite[Thm~6.2, Lem~6.3]{HolstTsogtgerel2013} we know that the Lichnerowicz equation has a solution, which by \cite[Thm~4.3]{HolstTsogtgerel2013} is unique. Thus the map
\[
\Psi: \mathcal P_{red}(\Sigma) \to \mathcal P_{CMC}(\Sigma), \qquad \Psi(\gamma,p) = (g,\pi)
\]
given by \eqref{eq-redvars} and \eqref{eq-lichn} is injective. Since $\mathcal P_{CMC}(\Sigma)$ consists of metrics with negative Yamabe invariant it is easy to see that $\Psi$ is also surjective. Furthermore, it can be shown that the map $\Psi$ is a diffeomorphism.

\subsection{The reduced Hamiltonian}\label{Sec-RedHam}

In order to define the reduced Hamiltonian, we begin with the standard approach of deriving the Hamiltonian from the gravitational action. For background see for example, \cite{Poisson2004}.
The standard expression for the gravitational action (Lagrangian) on $\mathcal V$ is
\begin{equation}\label{eq-GHY}
S(h) = 
\int_{\mathcal V}\scal_h \dv_h 
+ 2\int_{\partial \mathcal V} {\mathcal K}_h \,dS_h,
\end{equation}
where ${\mathcal K}_h$ is the mean curvature of $\partial V$. The surface integral is known as the Gibbons--Hawking--York boundary term and is necessary for the variational principle to recover the Einstein equations for a compact domain with boundary. 

For a spacetime of the form \eqref{eq:ADM_decomp}, standard textbook derivations (for example, \cite{Gourgoulhon}) 
yield the Hamiltonian
\begin{equation}\label{eq-Ham0}
\begin{split}
\mathcal H(\hg,\widehat\pi)
&=
-\int_\Sigma \left( N\Phi(\hg,\widehat\pi)+\widehat X^i\Phi_i(\hg,\widehat\pi) \right) \\
&\qquad 
-2\int_{\partial \Sigma} 
\left( NH_{\hg} - \hg_{ki}\hg_{lj}\widehat X^i\widehat{\overline\pi}^{kl}\widehat\nu^j \right) \,dS_{\hg},
\end{split}
\end{equation}
where $(\hg,\widehat\pi)$ is the initial data induced on $\Sigma$, $\Phi$ is the constraint map \eqref{eq-constraintmap}, $H_{\hg}$ is the mean curvature of $\partial \Sigma$ in $\Sigma$ with respect to the outward pointing unit normal $\widehat\nu$ with respect to $\hg$, and $\widehat{\overline\pi}$ is the de-densitized quantity defined by $\widehat{\overline\pi}=\widehat \pi \dv_{\hg}^{-1}$. Note that the boundary condition $g=\mathring g$ on $\partial\Sigma$ ensures that $\mathcal H$ defined by \eqref{eq-Ham0} generates the correct equations of motion. That is, $D\mathcal H_{(g,\pi)}(h,q)=-\int_\Sigma D\Phi_{(g,\pi)}^*(N,X)\cdot(h,q)$ from which it is readily checked that Hamilton's equations precisely agree with the Einstein evolution equation.
 
Through the Legendre transform, we can therefore write \eqref{eq-GHY} as 
\begin{equation} \label{eq-spacetimeaction}
\begin{split}
S(h) &= 
\int_{t_1}^{t_2} \Biggl\{ \int_{\Sigma}
\left(\widehat\pi^{ij}\partial_t \hg_{ij} 
+ N\Phi_0(\hg,\widehat\pi) + \widehat X^i\Phi_i(\hg,\widehat\pi)\right) \\
&\qquad\qquad
+2\int_{\p \Sigma}NH_{\hg} dS_{\hg}-2\int_{\p \Sigma} \hg_{ki}\hg_{lj}\widehat X^i\widehat{\overline\pi}^{kl}\widehat\nu^jdS_{\hg} \Biggr\}\,dt.
\end{split}
\end{equation}
Since we consider vacuum solutions of the Einstein equations, this expression reduces to
\begin{equation} \label{eq-vacuum-ham}
S(h) = 
\int_{t_1}^{t_2} \Biggl\{ 
\int_{\Sigma} \widehat\pi^{ij}\partial_t \hg_{ij}  
+ 2\int_{\p \Sigma} H_{\hg} dS_{\hg}
- 2\int_{\p \Sigma} \mathring{X}^i\widehat{\overline\pi}_{ij}\nu^jdS_{\hg} 
\Biggr\} \,dt,
\end{equation}
where we have also made use of the boundary conditions for $(N,\widehat X)$. We will next use \eqref{eq-redvars} to write the Hamiltonian \eqref{eq-vacuum-ham} in terms of the reduced variables $\gamma$ and $p$, which parametrise the CMC constraint manifold. We have
\[ \begin{split}
\widehat\pi^{ij}\partial_t \hg_{ij} 
&= 
t^{n-3} \varphi^{-4/(n-2)} p^{ij} \varphi^{4/(n-2)} \partial_t(t^2\gamma_{ij}) \\
&\qquad
+ t^{n-3}\varphi^{-4/(n-2)} p^{ij} \partial_t (\varphi^{4/(n-2)}) t^2 \gamma_{ij} \\
&\qquad 
- t^{n-3} (n-1)\varphi^{-4/(n-2)} \gamma^{ij} \partial_t(\hg_{ij}) \dv_{g} \\
&= \text{\big[since $\trace_\gamma p = 0$\big]} \\
&=
t^{n-1} p^{ij} \partial_t \gamma_{ij} 
- t^{n-3} (n-1) \varphi^{-4/(n-2)} \gamma^{ij} \partial_t(\hg_{ij}) \dv_{g} \\
&=
t^{n-1} p^{ij} \partial_t \gamma_{ij} 
- t^{-1} (n-1) \hg^{ij} \partial_t(\hg_{ij}) \dv_{\hg} \\
&=
t^{n-1} p^{ij} \partial_t \gamma_{ij} 
- t^{-1} (n-1) \hg^{ij} \partial_t(\hg_{ij}) \dv_{\hg} \\
&= \text{\bigg[since $\partial_t \dv_{\hg} = \frac{1}{2} \trace_{\hg} (\partial_t \hg) \dv_{\hg}$\bigg]} \\
&= 
t^{n-1}p^{ij} \partial_t \gamma_{ij} 
- 2t^{-1}(n-1) \partial_t\dv_{\hg} \\
&= 
t^{n-1}p^{ij}\p_t \gamma_{ij} + 2(n-1)\p_t(t^{-1}) \dv_{\hg} \\ 
&\qquad
-2(n-1)\p_t(t^{-1}\dv_{\hg})\\
&= 
t^{n-1}p^{ij}\p_t \gamma_{ij} 
-2(n-1)t^{n-2}\varphi^{2n/(n-2)} \dv_{\gamma} \\
&\qquad 
-2(n-1)\p_t(t^{-1}\dv_{\hg}).
\end{split} \]
Note here that the final term in the above expression above will only contribute a boundary term to \eqref{eq-spacetimeaction} at the $t=t_0$ and $t=t_1$ spacelike boundary components, so we discard this term without changing the equations of motion.

From \eqref{eq-vacuum-ham}, we write the Lagrangian as
\[\begin{split}
S(h) &=
\int_{t_1}^{t_2} \Biggl\{ 
\int_{\Sigma_t} 
\left( t^{n-1}p^{ij}\partial_t(\gamma_{ij}) - 2(n-1)t^{n-2}\varphi^{2n/(n-2)} \right)
\dv_{\gamma}\\
&\qquad\qquad\qquad
+ 2\int_{\p \Sigma_t} H_{\hg} \, dS_{\hg}
- 2\int_{\p \Sigma_t} \hg_{ki}\hg_{lj}t^{-1}\mathring X^i\widehat{\overline\pi}^{kl}\widehat\nu^j \, dS_{\hg} 
\Biggr\} \,dt,
\end{split} \]
We can directly read off that the conjugate momentum to $\gamma$ is $t^{n-1}p^{ij}$ and then by Legendre transform we have the reduced Hamiltonian
\begin{equation}\label{eq-RedHam0}
\begin{split}
\mathcal{H}^0_{red}(\gamma,p) 
&=
- 2\int_{\p \Sigma_t} H_{\hg} \, dS_{\hg} 
+ 2\int_{\p \Sigma_t} \hg_{ki}\hg_{lj}t^{-1}\mathring X^i\widehat{\overline\pi}^{kl}\widehat\nu^j \, dS_{\hg} \\
&\qquad
+ 2(n-1)t^{n-2}\int_{\Sigma_t} \varphi^{2n/(n-2)} \dv_{\gamma} \\
&=
t^{n-2}\Biggl(
- 2\int_{\p \Sigma_t}H_{g} \, dS_{g} 
+ 2\int_{\p \Sigma_t} g_{ki}g_{lj}\mathring X^i\overline\pi^{kl}\nu^j \, dS_{g} \\
&\qquad
+2(n-1)\int_{\Sigma_t} \dv_{g}\Biggr).
\end{split}
\end{equation}
We can add any constant to the Hamiltonian without changing the equations of motion. It is natural to use this freedom to ensure that the Hamiltonian for some fixed reference data evaluates to zero. We evaluate the reduced Hamiltonian at data corresponding to $\mathring g$, which by a mild abuse of notation is now taken to be any fixed metric with the same boundary conditions as $g$, and $\mathring \pi=-(n-1)\mathring g^{-1}dV_{\mathring g}$ then subtract this from \eqref{eq-RedHam0} to arrive at
\begin{equation} \label{eq-RedHamComp}
\begin{split}
&\mathcal H_{red}(\gamma,p) \\
&\qquad 
=
t^{n-2}\Biggl(2\int_{\p \Sigma_t}\left(H_{\mathring g}-H_{g}\right)\, dS_{\mathring g}+2(n-1)\int_{\Sigma_t}\left(\dv_{g}-\dv_{\mathring g} \right) \\
&\qquad\qquad\qquad\qquad
+2\int_{\partial \Sigma} 
\mathring g_{ki}\mathring g_{lj}\mathring X^i(\overline\pi^{kl}+(n-1)\mathring g^{kl})\nu^j \,dS_{\mathring g} \Biggr),
\end{split}
\end{equation}
This can be understood as a quasi-local volume-renormalised mass. Note that the choice of reference data being subtracted is modelled on the canonical leaves of the Milne model.

%%%%%%%%%%%%%%%%%%%%%%%%%%%%%%%%%%%%%%%%%%%%%%%%%%%%%%%%%%%%%%%%%%%%%%%%%%%%%%%%
\section{Hamiltonian Reduction for Asymptotically Milne-Like Spacetimes}
\label{sec-hamredmilne}
%%%%%%%%%%%%%%%%%%%%%%%%%%%%%%%%%%%%%%%%%%%%%%%%%%%%%%%%%%%%%%%%%%%%%%%%%%%%%%%%

We now develop a definition of the reduced Hamiltonian for a class of noncompact manifolds via a limiting process. In particular, we consider spacetimes that are asymptotic to a generalized Milne model. For this we first make some definitions.

\subsection{Milne-like spacetimes}

The Milne model is an early cosmological model obtained by a reparametrization of the interior of a lightcone in Minkowski space. In standard coordinates it is given by
\begin{equation}\label{eq-Milnemetric}
h = -dt^2+t^2 g_{\rm hyp} ,
\end{equation}
where $g_{\rm hyp}$ is the standard $n$-dimensional hyperbolic metric with constant negative curvature equal to $-1$. Each slice of constant $t$ is a hyperbolic metric with different curvature, all converging to the same section of future null infinity. The mean curvature of each such leaf is constant and equal to $-\frac{n}{t}$.

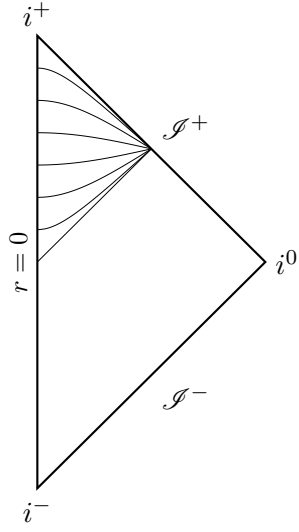
\begin{figure}[H]\label{fig-milne}
	\centering
	\begin{tikzpicture}[scale=3]

		\coordinate (O) at ( 0, 0);
		\coordinate (S) at ( 0,-1);
		\coordinate (N) at ( 0, 1);
		\coordinate (E) at ( 1, 0);
		\draw[thick] (N) -- (E) -- (S) -- cycle;

		\newcommand{\scri}{\mathscr{I}}
		\node[right] at (E) {$i^0$};
		\node[above] at (N) {$i^+$};
		\node[below] at (S) {$i^-$};
		\node[above, rotate=90] at (O) {\small{$r=0$}};
		\node[above right] at (0.5,0.5) {$\scri^+$};
		\node[below right] at (0.5,-0.5) {$\scri^-$};

		\tikzset{declare function={
				T(\t,\r)  = {\fpeval{1/pi*(atan(sqrt(\t**2+\r**2)+\r) + atan(sqrt(\t**2+\r**2)-\r))}};
				R(\t,\r)  = {\fpeval{1/pi*(atan(sqrt(\t**2+\r**2)+\r) - atan(sqrt(\t**2+\r**2)-\r))}};
		}}

		\message{Drawing time surfaces.^^J}
		\def\Nlines{6} 
		\newcommand\samp[1]{ tan(90*#1) } 
		\foreach \i [evaluate={\t=\i/(\Nlines+1);}] in {0,...,\Nlines}{
			\message{Drawing i=\i...^^J}
			\draw[line width=0.3,samples=30,smooth,variable=\r,domain=0.001:1]
			plot({ R(\samp{\t},\samp{\r}) }, { T(\samp{\t},\samp{\r}) });
		}
		
	\end{tikzpicture}
\caption{Penrose diagram for Minkowski space showing the Milne model foliated by hyperbolic metrics that all intersect the same section of $\mathscr{I}^+$.}
\end{figure}

It is straightforward to check that $h$ is still Ricci flat if we replace $g$ in \eqref{eq-Milnemetric} with any other Einstein metric satisfying $\ric_{g}=-(n-1)g$. That is, it is a solution to the vacuum Einstein equations. Note that this need not be a metric on $\mathbb R^n$. This motivates the following definition.

\begin{defn}
Let $(M,g)$ be a conformally compact $n$-dimensional Einstein manifold satisfying $\ric_{g}=-(n-1)g$. If $V=I\times M$ for some interval $I$ and
\[
h = -dt^2 + t^2 g
\]
we will say $(V,h)$ is a \textit{Milne-like spacetime}. 
\end{defn}
We remark that Milne-like spacetimes are continuously self-similar Ricci flat Lorentzian manifolds \cite{FischerMoncrief2002}.

We are interested in spacetimes that are asymptotic to Milne-like spacetimes in the following sense. 
\begin{defn}\label{defn-asympmilne}
Fix a Milne-like spacetime $(\mathring V,\mathring h)$ with associated Einstein manifold $(\mathring M,\mathring g)$. Consider another spacetime $(V=I\times M,h)$ given by the rescaled ADM decomposition \eqref{eq:ADM_decomp},
\begin{equation}\label{eq-asympmilne}
h = -N^2dt^2 + t^2g_{ij}(dx^i+t^{-1}X^idt)\otimes(dx^j+t^{-1}X^jdt),
\end{equation}
using the same $t$ coordinate on $I$. We say a spacetime is $(V,h)$ \textit{asymptotically Milne-like} if it is of the form \eqref{eq-asympmilne} and satisfies the following properties:
\begin{itemize}
\item There exists a diffeomorphism $\varphi:M\setminus K\to \mathring M\setminus \mathring K$ for compact sets $K$ and $\mathring K$,
\item $g-\varphi^*(\mathring g)\in C^{k,\alpha}_{\delta}$,
\item $N-1\in C^{k,\alpha}_\delta$,
\item $X \in C^{k,\alpha}_\delta$,
\item $\partial_t g \in C^{k,\alpha}_\delta$.
\end{itemize}
\end{defn}

The condition that $X$ decays at infinity is required so that the final term in \eqref{eq-RedHamComp} is finite when $\Sigma$ is taken to be an APE manifold and the boundary term is defined in an appropriate limiting sense. Furthermore, this combined with the condition on $\partial_t g$ imposes a decay condition on $\pi$ via the evolution equation \eqref{eq-evo}. Specifically, the data $(g,\pi)$ determined by the $t=1$ slice is APE initial data in the sense of Definition~\ref{def-APEdata}, with the same assumptions on $\delta,k$ and $\alpha$. That is, $\pi+(n-1)\mathring g^{-1}dV_{\mathring g}\in C^{k-1,\alpha}_{\delta}$. Moreover, these decay conditions are those required for $\m_{VR,\mathring g}(g,\pi)$ to be well-defined, see Theorem~\ref{thm-well-defined}.

With the above definitions in mind, we define the reduced phase space in the non-compact setting as
\begin{align*}
	\mathcal P_{red}(M)&=\{(\gamma,p)\mid(\gamma-\varphi^*(\mathring g))\in C^{k,\alpha}_\delta(S^2_+(T^*M)), \scal_{g}=-n(n-1),\\
	&\quad p\in C^{k-1,\alpha}_\delta(S^2(TM)\otimes \Lambda^n),\trace_{\gamma}(p)=0, \divergence_{\gamma}(p)=0   \}.
\end{align*}

\subsection{Normalization and limiting process}

We now would like to develop the reduced Hamiltonian for asymptotically Milne-like spacetimes by considering the reduced Hamiltonian developed in Section~\ref{Sec-CompactHam} on bounded domains in such spacetimes and taking a limit. We fix a Milne-like reference spacetime $(V=I \times M,\mathring h)$ with
\[
\mathring h = -dt^2+t^2\mathring g.
\]
Consider an exhaustion $M=\bigcup_{r\to\infty}\Sigma_r$ of $M$ by compact manifolds $\Sigma_r$ with boundary $\p\Sigma_r$ so that near infinity the boundaries $\p\Sigma_r$ foliate the asymptotic end. For each $\Sigma_r$ we can construct a reduced phase space and reduced Hamiltonian as in Section~\ref{Sec-CompactHam} using $\mathring g|_{\Sigma_r}$ as the reference metric. We first show that the mean curvature boundary term can be replaced with an ADM-like mass integral. For metrics $g$ and $\mathring g$ with $g|_{\p \Sigma_r} = \mathring g|_{\p \Sigma_r}$ we define
\[
\m_{ADM,\mathring g}(g,\partial \Sigma_r)
=
\int_{\partial \Sigma_r}
\left(\mathring g^{ij}\mathring\nabla_i g_{jk}-\mathring\nabla_k (\mathring g^{ij}g_{ij})\right)\nu^k 
\,dS_{\mathring g},
\] 

\begin{lem}\label{lem-ADMlim}
Let $(\mathring M,\mathring g)$ be an APE manifold with an exhaustion $\bigcup_{r\to\infty}\Sigma_r$ of $\mathring M$ as above. For sufficiently large $r$ consider a family of metrics $g=g(r)$ such that 
$g|_{\p \Sigma_r} = \mathring g|_{\p \Sigma_r}$. Then
\[ 
\m_{ADM,\mathring g}(g,\partial \Sigma_r)=2\int_{\partial \Sigma_r}(H_{\mathring g}-H_g)dS_g.
\]

In particular, along the foliation given by the exhaustion, we have
\[ 
\lim_{r\to\infty}
\left(
2\int_{\partial \Sigma_r}(H_{\mathring g}-H_g)dS_g-\m_{ADM,\mathring g}(g,\partial \Sigma_r)
\right)
=0.
\]
\end{lem}

\begin{proof}
This is a standard computation for asymptotically flat manifolds, and the same arguments apply here. We follow the arguments of Hawking and Horowitz \cite{hawking1996gravitational}.

We first fix $r$ large and work in coordinates adapted to the foliation provided by the given exhaustion, namely we write
\[
g=dr^2+q,\qquad \text{ and}\qquad \mathring g=dr^2+\mathring q,
\] 
where $q$ and $\mathring q$ are metrics on $\partial \Sigma_r$. We can then directly compute
\[\begin{split}
2(H_{\mathring g}-H_g)
&=
\mathring q^{AB}\p_r \mathring q_{AB}-q^{AB}\p_r q_{AB}\\
&=
q^{AB}\p_r\left(\mathring q_{AB}-q_{AB} \right)+\left( \mathring q^{AB}-q^{AB} \right)\p_r \mathring q_{AB}\\
&=
q^{AB}\p_r\left( \mathring q_{AB}-q_{AB} \right),
\end{split}\]
where we used the condition $\mathring q=q$ on $\partial\Sigma_r$ in the last step. This can now be shown to be equal to $\m_{ADM,\hg}(g,\partial \Sigma_r)$ as follows. We compute
\[ \begin{split}
\m_{ADM,\mathring g}(g,\partial \Sigma_r)
&=
\int_{\partial \Sigma_r} 
\left(\mathring g^{ij}\mathring\nabla_i (g_{jk}-\mathring g_{jk})-\mathring g^{ij}\mathring\nabla_k (g_{ij}-\mathring g_{ij})
\right)\nu^k \,dS\\
&=
\int_{\partial \Sigma_r} \Bigg( 
\mathring g^{ij}\mathring\nabla_i ((g_{jk}-\mathring g_{jk})\nu^k)-\mathring g^{ij}(g_{jk}-\mathring g_{jk})\mathring\nabla_i \nu^k\\
&\qquad\qquad
-\mathring g^{ij}\mathring\partial_r (g_{ij}-\mathring g_{ij})+2\mathring\Gamma^p_{kj}(g_{ip}-\mathring g_{ip})\mathring g^{ij}\nu^k \Bigg) \,dS,
\end{split} \]
and then note that the second and fourth terms vanish since $g=\mathring g$ on $\partial\Sigma_r$, while the first term vanishes since $(g_{jk}-\mathring g_{jk})\nu^k=0$. Since $g-\mathring g$ is only nonzero in directions tangential to $\p\Sigma_r$ we find that
\[ \begin{split}
\m_{ADM,\mathring g}(g,\partial \Sigma_r)
&=
\int_{\partial \Sigma_r}-\mathring q^{AB}\mathring\partial_r (q_{AB}-\mathring q_{AB})\,dS\\
&=
2\int_{\p\Sigma_r}(H_{\mathring g}-H_g)\,dS.
\end{split} \]
\end{proof}

Now for each $\Sigma_r$ we have a reduced Hamiltonian from \eqref{eq-RedHamComp} for a region of spacetime $\Sigma_{r}\times I$. Taking $r\to\infty$, and making use of Lemma~\ref{lem-ADMlim} and the decay conditions for $X$ and $\pi$, we arrive at the main conclusion on this section. Namely, the reduced Hamiltonian on $M$ is given by
\[
\mathcal H_{red}(\gamma,p) = t^{n-2}\m_{VR,\mathring g}(g),
\]
 where $\m_{VR,\mathring g}(g)$ is the volume-renormalized mass of $g$ with respect to $\mathring g$. By Theorem \ref{thm-well-defined}, Remark \ref{rem-compatibility} and the fact we have a CMC foliation, it is clear that $\mathcal H_{red}$ is well-defined on the reduced phase space $\mathcal P_{red}$

\begin{prop}
The lapse $N$ is uniquely determined by the CMC slicing condition.
\end{prop}

\begin{proof}
This follows from the Einstein evolution equations for vacuum initial data. In terms of the induced metric $\widehat g$ and second fundamental form $K$ the well-known (see, for example \cite{Gourgoulhon}) vacuum evolution equation are
\begin{equation}
\begin{split}\label{eq-evoeqs}
\frac{\partial}{\partial t} \widehat g_{ij} 
&= -2N K_{ij} + \mathcal{L}_X \widehat g_{ij} , \\
\frac{\partial}{\partial t} K_{ij} 
&= -\widehat \nabla_i \widehat \nabla_j N 
+ N \left( \widehat R_{ij} - 2 K_{ik} K^k{}_j + \trace_{\widehat g}K K_{ij} \right) 
+ \mathcal{L}_X K_{ij},
\end{split}
\end{equation} 
where we raise and lower indices with $\widehat g$. If we impose the CMC slicing condition $\trace_{\widehat g}(K)=\tau=-n/t$, we have
\[
\frac{\partial}{\partial t}(\trace_{\widehat g}(K))
= \frac{n}{t^2}
= -K^{ij}\frac{\partial}{\partial t}(\widehat g_{ij})+\widehat g^{ij}\frac{\partial }{\partial t}(K_{ij}).
\]
Making use of \eqref{eq-evoeqs} and the fact that $\mathcal L_X(\widehat g^{ij}K_{ij})=0$ we arrive at the elliptic equation
\begin{equation}\label{eq-elliptic-N}
\widehat \Delta N -N|K|_{\widehat g}^2 = -\frac{\tau^2}{n},
\end{equation}
which has a unique solution with $N \to 1$ at infinity.
\end{proof} 

\begin{rem}
To fully reduce the phase space so that each point in the phase space corresponds precisely to a unique solution to the Einstein equations, we must also fix the shift $X$ and quotient out by diffeomorphisms of $M$ that are isotopic to the identity. Fixing $X$ will play a role in analysing the second variation of the reduced Hamiltonian in the following section, but we do not need to quotient out by diffeomorphisms of $M$ since the quantities we are interested in are diffeomorphism invariant.
\end{rem}

%%%%%%%%%%%%%%%%%%%%%%%%%%%%%%%%%%%%%%%%%%%%%%%%%%%%%%%%%%%%%%%%%%%%%%%%%%%%%%%%
\section{Variation of the reduced Hamiltonian}\label{Sec-variations}
%%%%%%%%%%%%%%%%%%%%%%%%%%%%%%%%%%%%%%%%%%%%%%%%%%%%%%%%%%%%%%%%%%%%%%%%%%%%%%%%

\subsection{First variation}

Since the reduced Hamiltonian is proportional to the volume-renormalized mass of $g$ with respect to $\mathring g$, it is known that critical points with respect to variations in $g$ that preserve scalar curvature are exactly Einstein metrics \cite{DKM2023arxiv}. We will now show that critical points with respect to the reduced variables on the entire phase space are also Einstein metrics with vanishing reduced momentum.

\begin{thm}\label{thm-firstvar}
	Critical points of $\mathcal H_{red}$ on $\mathcal P_{red}$ are precisely $(\gamma,p)$ where $\gamma$ satisfies $\ric_\gamma=-(n-1)\gamma$ and $p=0$.
\end{thm}

\begin{proof}
We first calculate the reduced Hamiltonian in terms of $\gamma$ and the conformal factor $\varphi$ that solves the Lichnerowicz equation, noting that
\begin{equation} \label{eq-confVRmass}
\mathcal H_{red}(\gamma,p)=t^{n-2}\m_{VR,\mathring g}(\varphi^{\frac{4}{n-2}}\gamma).
\end{equation}
We follow the computations for the proof of the conformal positive mass theorem for the volume-renormalized mass, Theorem~4.5 of \cite{DKM2023arxiv}.
We have
\[ \begin{split}
\m_{VR,\mathring g}(\varphi^{\frac{4}{n-2}}\gamma)
&=
\lim_{r\to\infty}\Big( \int_{\partial \Sigma_r}  \varphi^{\frac{4}{n-1}}\left(\mathring g^{ij}\mathring \nabla_i(\gamma_{jk})-\mathring\nabla_k(\mathring g^{ij}\gamma_{ij}) \dv \right)dS_{\mathring g}\\
&\qquad
+\int_{\Sigma_r} \frac{4}{n-2}\varphi^{\frac{6-n}{n-2}} \left(\mathring g^{ij}\gamma_{jk}\partial_i(\varphi)-\mathring g^{ij}\gamma_{ij}\partial_k(\varphi)\right)\,dS_{\mathring g}\\
&\qquad
+2(n-1)\int_{\Sigma_r}(\varphi^{\frac{2n}{n-2}}dV_{\gamma}-dV_{\mathring g})\Big) \\
&=
\lim_{r\to\infty}\Big( \int_{\Sigma_r} \left(\mathring g^{ij}\mathring \nabla_i(\gamma_{jk})-\mathring\nabla_k(\mathring g^{ij}\gamma_{ij}) \dv \right)dS_{\mathring g} \\
&\qquad
+2(n-1)\left(\int_{\Sigma_r}\varphi^{\frac{2n}{n-2}}\left( dV_\gamma -dV_{\mathring g}\right)+ \int_{\Sigma_r}(\varphi^{\frac{2n}{n-2}}-1)dV_{\mathring g}  \right)\\
&\qquad
-\frac{4(n-1)}{(n-2)}\int_{\partial \Sigma_r}\varphi^{\frac{6-n}{n-2}}\partial_i(\varphi)\nu^i\,dS_{\mathring g}\Big) \\
&=
\m_{VR,\mathring g}(\gamma)+2(n-1)\int_M\Big(\varphi^{\frac{2n}{n-2}}-1-\frac{2}{n-2}\Delta_{\mathring g}(\varphi)\Big)dV_{\mathring g} .
\end{split} \]

We first compute the variation with respect to $p$, which gives
\begin{equation} \label{eq-varp}
D_p\mathcal H(\gamma,p)(r)
= t^{n-2}2(n-1) \int_M \Big(
\frac{2n}{n-2}\varphi^{\frac{n+2}{n-2}}\varphi_p - \frac{2}{n-2}\Delta_{\mathring g} \varphi_p 
\Big) dV_{\mathring g},
\end{equation}
where $\varphi_p=D_p(\varphi)(r)$ is used to denote the variation of $\varphi$ in the direction of $r$.

Following Fischer and Moncrief in the compact case \cite[Sec~4.2]{FischerMoncrief2002}, we can calculate $\varphi_p$ from the Lichnerowicz equation, which when expressed in terms of $p$ is given by
\[
-\frac{4(n-1)}{n-2}\Delta_\gamma \varphi 
- n(n-1)\varphi
+ n(n-1)\varphi^{\frac{n+2}{n-2}} 
- |p|_\gamma^2\varphi^{-\frac{3n-2}{n-2}}dV_\gamma^{-2}
= 0,
\]
where $dV_\gamma^{-2}$ denotes a scalar density of weight $-2$, which ``de-densitizes" $|p|^2_\gamma$; that is, $|p|^2dV_\gamma^{-2}=|\overline p|_\gamma^2$. Taking the variation with respect to $p$ gives
\begin{equation} \label{eq-VarLich}
\begin{split}
&-\frac{4(n-1)}{n-2}\Delta_\gamma \varphi_p 
- n(n-1)\varphi_p
+ \frac{n(n-1)(n+2)}{n-2}\varphi^{\frac{4}{n-2}}\varphi_p \\ 
& \qquad
- 2(p\cdot_\gamma r) \varphi^{\frac{2-3n}{n-2}}dV_\gamma^{-2}
+ \frac{3n-2}{n-2}|p|_\gamma^2\varphi^{\frac{4-4n}{n-2}}\varphi_pdV_\gamma^{-2}= 0,\end{split}
\end{equation}
We first set $p=0$ and show this implies $\varphi_p=0$ by multiplying \eqref{eq-VarLich} by $\varphi_p$ and integrating over $M$, yielding
\begin{equation} \label{eq-premaxprinc}
\int_M  4|\nabla_\gamma\varphi_p|_\gamma^2 \dv_\gamma
= \int_M\left(\frac{n}{n-2}- n(n+2)\varphi^{\frac{4}{n-2}}\right)\varphi_p^2 dV_\gamma,
\end{equation}
where a surface integral at infinity from the divergence theorem vanishes due to the decay of $\varphi_p$. By the maximum principle applied to the Lichnerowicz equation we obtain $\varphi^{\frac{4}{n-2}}\geq 1$, which implies that
\[
\left(\frac{n}{n-2}- n(n+2)\varphi^{\frac{4}{n-2}}\right)<0.
\]
In particular, \eqref{eq-premaxprinc} can only hold if $\varphi_p$ is identically zero, which in turn implies $D_p\mathcal H=0$ by \eqref{eq-varp}.

We next show the converse. That is, if $D_p\mathcal H=0$ at some $(\gamma,p)$ then $p$ must vanish. To this end, we take $r=p$ and will show $\varphi_p=0$.

The maximum principle applied to \eqref{eq-VarLich} shows that at any minimum value of $\varphi_p$ we must have
\[ \begin{split}
&\left( \frac{n+2}{n-2}\varphi^{\frac{4}{n-2}}-1+\frac{3n-2}{n(n-1)(n-2)}\left|p\right|^2_\gamma dV_\gamma^{-2} \varphi^{\frac{4-4n}{n-2}} \right)\varphi_p \\
&\qquad\qquad
\geq \frac2{n(n-1)}\left|p\right|^2_\gamma\dv_\gamma^{-2}\varphi^{\frac{2-3n}{n-2}}.
\end{split} \]
Since $\varphi^{\frac{4}{n-2}}\geq1$, the term in parentheses must be positive and therefore we must have $\varphi_p\geq0$ at any minimum value. That is, $\varphi_p\geq0$ everywhere. It then follows from equation \eqref{eq-varp} and the fact that $p$ is critical that we have
\[
\int_M \Delta_{\mathring g} \varphi_p \dv_{\mathring g}>0,
\] 
unless $\varphi_p$ is identically zero. However we now show $\int_M \Delta_{\mathring g} \varphi_p \dv_{\mathring g}=0$. For this purpose, consider the operator
\begin{align*}
	P_{\gamma}&=	-\frac{4(n-1)}{n-2}\Delta_\gamma  
	- n(n-1)
	+ \frac{n(n-1)(n+2)}{n-2}\varphi^{\frac{4}{n-2}}\\ &\qquad+\frac{3n-2}{n-2}|p|_\gamma^2\varphi^{\frac{4-4n}{n-2}}(dV_\gamma)^{-2}
\end{align*}
so that we can rewrite equation \eqref{eq-VarLich} as
\begin{align*}
	P_{\gamma}\varphi_p=2|p|_{\gamma}^2\varphi^{\frac{2-3n}{n-2}}(dV_\gamma)^{-2}.
\end{align*}
Since $\varphi\geq 1$ and $|p|_\gamma^2\geq 0$, we have
\begin{align*}
	P_{\gamma}\geq -\frac{4(n-1)}{n-2}(\Delta_\gamma-n)
\end{align*}
in the $L^2$ sense, so that the $L^2$-kernel of $P_{\gamma}$ is trivial. Furthermore, since $\varphi\to1$ at infinity, $P_{\gamma}$ has the same indicial roots as $\Delta_{\gamma}-n$ and is therefore an isomorphism as an operator \cite[Thm~C, Prop~E]{Lee06}
\begin{align*}
	P_{\gamma}:C^{k,\alpha}_{2\delta}\to C^{k-2,\alpha}_{2\delta}
\end{align*}
for $2\delta\in (-1,n)$. In particular, since $|p|^2_{\gamma}\in C^{0,\alpha}_{2\delta}$ for some $2\delta>n-1$, we get $\varphi_p\in C^{2,\alpha}_{2\delta}$. Applying the divergence theorem on a large ball $B_R$ and letting $R\to\infty$ shows that
\begin{align*}
	\int_M\Delta_{\gamma}\varphi_p\dv_{\gamma}=0.
\end{align*}
We therefore conclude that $\varphi_p \equiv 0$. Substituting this back into \eqref{eq-VarLich} and recalling that we set $r=p$, one immediately concludes that $p\equiv0$ at a critical point.

Note that if $p\equiv0$ then the unique solution to the Lichnerowicz equation is the constant solution $\varphi\equiv1$. In this case, the reduced Hamiltonian is simply
\[ 
\mathcal H_{red}(\gamma,p)=t^{n-2}\m_{VR,\mathring g}(\gamma).
\] 
It therefore follows that critical points of $\mathcal H_{red}$ are precisely critical points of $\m_{VR,\mathring g}$ subjected to the restriction $\scal=-n(n-1)$, which is known to be exactly the metrics $\gamma$ with $\ric_\gamma=-(n-1)\gamma$, see \cite{DKM2023arxiv, McCormick2025}.
\end{proof}

\subsection{Second variation}

We next calculate the Hessian of the reduced Hamiltonian at a critical point. As with the preceding section, this follows the compact case studied in \cite[Sec~4.3]{FischerMoncrief2002}. Before stating the main result of this section, we recall the Einstein operator which acts on symmetric $2$-tensors $h$ by
\begin{equation*}
\Delta_E h_{ij} 
=
- \widetilde{\nabla}_k\widetilde \nabla^k h_{ij} 
+ 2R_{\widetilde\gamma;\,ikjl}h^{kl}.
\end{equation*}

\begin{thm}
	Let $(\widetilde \gamma, 0)$ be a critical point of $\mathcal H_{red}$ on $\mathcal P_{red}$. The second variation of the volume-renormalized mass at $(\widetilde\gamma,0)$ is given by
\begin{equation*} 
		\begin{split}
			&D^2(\m_{VR,\mathring g}(\varphi^{\frac{4}{n-2}}\gamma))((h,r),(h,r)) \\
			&\quad=
			\int_M \left( \frac12 h^{ij}\Delta_E h_{ij} + 2|r|^2_{\widetilde\gamma}dV_{\widetilde\gamma}^{-2} 
			\right) dV_{\widetilde\gamma}.
		\end{split}
	\end{equation*}
\end{thm}
\begin{proof}
Fix a critical point of the Hessian $(\gamma,p)=(\widetilde\gamma,0)$, where $\widetilde\gamma$ is an Einstein metric satisfying $\ric_{\widetilde \gamma}=-(n-1)\widetilde\gamma$. Let $(h,r)\in T_{(\widetilde\gamma,0)}\mathcal P_{red}$ be a perturbation at the critical point.

Following a similar computation as that which led to \eqref{eq-confVRmass}, we express the volume-renormalized mass as
\begin{equation} \label{eq-VRmass-2ndvar}
\begin{split}
&\m_{VR,\mathring g}(\varphi^{\frac{4}{n-2}}\gamma) \\ 
&\quad = 
\lim_{r\to\infty} \Bigg( \int_{\partial\Sigma_r} 
\left( \mathring g^{ij}\mathring \nabla_i \gamma_{jk}
-\mathring\nabla_k(\mathring g^{ij}\gamma_{ij}) \right) \nu^k dS_{\mathring g}\\
& \quad \quad
-2(n-1)\int_{\Sigma_r} dV_{\mathring g}
+2(n-1)\int_{\Sigma_r} \left( \varphi^{\frac{2n}{n-2}} -\frac{2}{n-2}\Delta_\gamma \varphi \right)  dV_\gamma \Bigg),
\end{split}
\end{equation}
where we have used the asymptotics to interchange metrics in surface integrals at infinity prior to using the divergence theorem. We first examine the second variation of the last term in \eqref{eq-VRmass-2ndvar},
\begin{equation} \label{eq-2ndvar1}
D^2\left( \int_M\left( \varphi^{\frac{2n}{n-2}}-\frac{2}{n-2}\Delta_{\gamma} \varphi \right) dV_{\gamma}\right)\left((h,r),(h,r)\right).
\end{equation}
Let $\delta\varphi=D\varphi(h,r))$, which can be shown to vanish at a critical point as follows. Since $p=0$ at a critical point, the unique solution $\varphi$ to the Lichnerowicz equation \eqref{eq-lichn} is constant $\varphi\equiv 1$, so in linearizing \eqref{eq-lichn} around a critical point gives
\begin{equation*}
	-\frac{4(n-1)}{n-2}\Delta_{\widehat\gamma}(\delta\varphi)+\frac{4n(n-1)}{n-2}\delta\varphi =0,
\end{equation*}
which implies $\delta\varphi=0$. We next let $\delta^2\varphi=D^2\varphi((h,r),(h,r))$ denote the second variation of $\varphi$ and making use of the fact that $\delta\varphi=0$, we calculate the expression \eqref{eq-2ndvar1} to be
\begin{equation} \label{eq-2ndvar2}
\begin{split}
& \int_M \left(\frac{2n}{n-2}\delta^2\varphi-\frac{2}{n-2}D^2\left(\Delta_{\gamma} \varphi \right)
((h,r),(h,r))\right)dV_{\gamma} \\
&\qquad 
+\int_M D^2\left( dV_{\gamma}\right)\left((h,r),(h,r)\right),
\end{split}
\end{equation}
recalling that $\varphi \equiv 1$ at a critical point. Note that the second variation of $\Delta_\gamma \varphi$ can be computed from the second variation of the Lichenerowicz equation, which at a critical point reduces to
\[
-\frac{2}{n-2} D^2 \left(\Delta_{\widetilde \gamma}\varphi \right)((h,r)(h,r))
+ \frac{2n}{(n-2)}\delta^2\varphi
= \frac{|r|^2_{\widetilde\gamma}}{n-1}(dV_{\widetilde\gamma})^{-2}.
\] 
Comparing this to \eqref{eq-2ndvar2} we see that we now only must compute the second variation of the volume form. We continue precisely as in \cite[Sec~4.3]{FischerMoncrief2002}. Recall that to fully reduce the phase space we must also quotient out by diffeomorphisms, which until now has not played a role. However, we now take advantage of that freedom by imposing the gauge condition relative to $\widetilde\gamma$,
\[
\widetilde \gamma^{jk}\widetilde\nabla_j \gamma_{kl} = 0,
\]
for metrics $\gamma$ in the reduced phase space. In particular, the perturbations $h$ that we consider must be divergence-free with respect to $\widetilde \gamma$. Since $h$ must also preserve the constant scalar curvature, we have
\[\begin{split}
D\scal_{\widetilde\gamma}(h)
&=
\widetilde\nabla^i\widetilde\nabla^j h_{ij} - \widetilde\Delta(\trace_{\widetilde \gamma}(h))
-\ric_{\widetilde\gamma}^{ij}h_{ij} \\
&=
\left(-\widetilde\Delta+(n-1)\right)\trace_{\widetilde\gamma}(h)=0,
\end{split} \]
since $\widetilde\gamma$ is Einstein. In particular, this implies that $\trace_{\widetilde\gamma}(h)$ vanishes and therefore that $h$ is transverse-traceless.

Now consider a curve of metrics $\gamma(\lambda)$ with $\gamma(0)=\widetilde\gamma$ and $\gamma'(0)=h$, and set $\ell=\gamma''(0)$. Then 
\[\begin{split}
D^2(dV_{\gamma})(h,h)&=\frac{d^2}{d\lambda^2}(dV_{\gamma(\lambda)})|_{\lambda=0} \\
&=
\frac{d}{d\lambda}
\left( \frac12\trace_{\gamma(\lambda)}(\gamma'(\lambda))dV_{\gamma(\lambda)} \right)|_{\lambda=0} \\
&=
\frac12\left( \trace_{\widetilde\gamma}(\ell)-|h|^2_{\widetilde\gamma} \right),
\end{split} \]
where we make use of the fact that $\trace_{\widetilde\gamma}(h)=0$. Putting this together into \eqref{eq-VRmass-2ndvar} and using the same curve of metrics to calculate the second variation of the surface integral at infinity, we find that
\begin{equation} \label{eq-VRmass-2ndvar-2}
\begin{split}
&D^2(\m_{VR,\mathring g}(\varphi^{\frac{4}{n-2}}\gamma))((h,r),(h,r)) \\
&\qquad=
\lim_{r\to\infty}\int_{\partial\Sigma_r} 
\left(\mathring g^{ij}\mathring \nabla_i \ell_{jk} 
-\mathring\nabla_k(\mathring g^{ij}\ell_{ij}) \right)\nu^k dS_{\mathring g}\\
&\qquad\qquad
+2\int_M |r|^2_{\widetilde\gamma}dV_{\widetilde\gamma}^{-2}\,dV_{\widetilde\gamma}
+ (n-1)\int_M \left( \trace_{\widetilde\gamma}(\ell)
-|h|^2_{\widetilde\gamma} \right)dV_{\widetilde\gamma} \\
&\qquad=
\int_M\Big( \widetilde\nabla^i\widetilde\nabla^j \ell_{ij} 
-\widetilde\Delta\ell + (n-1)\trace_{\widetilde\gamma}(\ell)\\
&\qquad\qquad
+2|r|_{\widetilde\gamma}^2dV_{\widetilde\gamma}^{-2}
-(n-1)|h|^2_{\widetilde\gamma}\Big)dV_{\widetilde\gamma} .
\end{split}
\end{equation}

It remains to evaluate the terms containing $\ell$, for a curve of metrics preserving scalar curvature and the gauge condition. To this end, we note that the combination of terms involving $\ell$ in \eqref{eq-VRmass-2ndvar-2} are precisely the linearized scalar curvature in the direction of $\ell$, that is
\[
D\scal_{\widetilde\gamma}(\ell)
= \widetilde\nabla^i\widetilde\nabla^j \ell_{ij} - \widetilde\Delta\ell
+(n-1)\trace_{\widetilde\gamma}(\ell).
\]
Since scalar curvature is constant on the curve of metrics $\gamma(\lambda)$ we have
\[
\frac{d^2}{d\lambda^2}(\scal_{\gamma(\lambda)})|_{\lambda=0}
= D\scal_{\widehat\gamma}(\ell) + D^2\scal_{\widetilde\gamma}(h,h)=0,
\]
where
\[ 
D^2\scal_{\widetilde\gamma}(h,h)
= \frac12 h^{ij}\widetilde\nabla^k\widetilde\nabla_k h_{ij}
-R^{ijkl}_{\widetilde\gamma}h_{jl}h_{ik}
-(n-1)|h|^2_{\widetilde\gamma}
\] 
is the second variation of scalar curvature, making use of the fact that $\widetilde\gamma$ is Einstein and $h$ is transverse-traceless. Bringing everything together we arrive at
\begin{equation}
\begin{split}
&D^2(\m_{VR,\mathring g}(\varphi^{\frac{4}{n-2}}\gamma))((h,r),(h,r)) \\
&\quad=
\int_M \left( -D^2\scal_{\widetilde\gamma}(h,h)+2|r|_{\widetilde\gamma}^2dV_{\widetilde\gamma}^{-2}-(n-1)|h|^2_{\widetilde\gamma} \right)dV_{\widetilde\gamma} \\ 
&\quad=
\int_M \left( \frac12 \widetilde\nabla_{k}(h_{ij})\widetilde\nabla^{k}(h^{ij})
+ R_{\widetilde\gamma}^{ijkl}h_{ik}h_{jl}
+ 2|r|^2_{\widetilde\gamma}dV_{\widetilde\gamma}^{-2} 
\right) dV_{\widetilde\gamma}\\
&=\int_M \left( \frac12 h^{ij}\Delta_E h_{ij} + 2|r|^2_{\widetilde\gamma}dV_{\widetilde\gamma}^{-2} 
\right) dV_{\widetilde\gamma}.
\end{split}
\end{equation}
\end{proof}
This implies the following immediate corollary.
\begin{cor}\label{cor-secondvar}
		Let $(\widetilde \gamma, 0)$ be a critical point of $\mathcal H_{red}$ on $\mathcal P_{red}$, define  
		\[
		\lambda_{\rm min}
		= \inf\frac{\int_M h^{ij}\Delta_E h_{ij} \dv_{\widetilde \gamma}}{\|h\|_{L^2(M,\widetilde\gamma)}^2},
		\]
		where the infimum is taken over $0\neq h\in C^{k,\alpha}_{\delta}(S^2(T^*M))$. Then the second variation of the volume-renormalized mass satisfies
		\begin{equation*}
			D^2(\m_{VR,\mathring g}(\varphi^{\frac{4}{n-2}}\gamma))((h,r),(h,r)) \\
			\geq \frac12 \lambda_{\rm min} \|h\|_{L^2(M,\widetilde\gamma)}^2+2\|r\|_{L^2(M,\widetilde\gamma)}^2.
		\end{equation*}
\end{cor}

%%%%%%%%%%%%%%%%%%%%%%%%%%%%%%%%%%%%%%%%%%%%%%%%%%%%%%%%%%%%%%%%%%%%%%%%%%%%%%%%
\section{Monotonicity of the Volume-Renormalized Mass}\label{Sec-mono}
%%%%%%%%%%%%%%%%%%%%%%%%%%%%%%%%%%%%%%%%%%%%%%%%%%%%%%%%%%%%%%%%%%%%%%%%%%%%%%%%

We now show that Einstein evolution equations seek to minimise the volume-renormalized mass in the sense that the reduced Hamiltonian is monotonically decreasing unless the mass is zero, in which case it remains constant along the flow. We will make the choice $X=0$ throughout this section.

\begin{thm}\label{thm-mono}
	The volume-renormalized mass is non-increasing under Einstein evolution, and is constant if and only if the resulting spacetime is Milne-like.
\end{thm}

\begin{proof}
We make use of the evolution equations \eqref{eq-evoeqs} to first calculate
\[ \begin{split}
\frac{\partial }{\partial t}(dV_{g})&=\frac{\partial }{\partial t}(dV_{t^{-2}\widehat g})\\
&=
\frac12\left( g^{ij}t^{-2}(-2NK_{ij})-2nt^{-1} \right)dV_g \\
&=nt^{-1}(N-1)dV_g
\end{split} \]
We then can can directly compute
\[\begin{split}
\frac{\partial}{\partial t}(\m_{VR,\mathring g}(g))
&=\lim_{R\to\infty}\Big(
\int_{\p B_R} \left(\mathring{g}^{ij}\onab_{i}(\frac{\partial }{\partial t}t^{-2}\widehat g_{jk})-\mathring{g}^{ij}\onab_k(\frac{\partial }{\partial t}t^{-2}\widehat g_{ij})\right)\nu_{\mathring{g}}^k\,dS_{\mathring g}\\
&\quad+2(n-1)nt^{-1}\int_{B_R}(N-1)\,dV_g\Big)\\
&=\lim_{R\to\infty}\Big(
t^{-2}\int_{\p B_R} \left(\mathring{g}^{ij}\onab_{i}(\frac{\partial }{\partial t}\widehat g_{jk})-\mathring{g}^{ij}\onab_k(\frac{\partial }{\partial t}\widehat g_{ij})\right)\nu_{\mathring{g}}^kdS_{\mathring g}\\
&\quad
-2t^{-1}\m_{ADM,\mathring g}(g,R)+2(n-1)nt^{-1}\int_{B_R}(N-1)dV_g\Big) \\
&=\lim_{R\to\infty}\Big(
-2 t^{-2}\int_{\p B_R}
 \left(\mathring{g}^{ij}\onab_{i}(NK_{jk})-\mathring{g}^{ij}\onab_k(NK_{ij})\right)\nu_{\mathring{g}}^kdS_{\mathring g}\\
&\quad
-2t^{-1}\m_{ADM,\mathring g}(g,R)+2(n-1)nt^{-1}\int_{B_R}(N-1)dV_g \Big)\\
&=\lim_{R\to\infty}\Big(
-2t^{-2}\int_{\p B_R}
 \Big(\mathring{g}^{ij}\onab_{i}(N)K_{jk}+\mathring{g}^{ij}\onab_{i}(K_{jk})N\\
&\quad
-\mathring{g}^{ij}\onab_k(N)K_{ij}-\mathring{g}^{ij}\onab_k(K_{ij})N\Big)\nu_{\mathring{g}}^kdS_{\mathring g}\\
&\quad -2t^{-1}\m_{ADM,\mathring g}(g,R)+2(n-1)nt^{-1}\int_{B_R}(N-1)dV_g \Big)\\
&=\lim_{R\to\infty}\Big(
-2t^{-2}\int_{\p B_R}
 \Big(-t\partial_{k} N + g^{ij}\onab_{i}(K_{jk}+t\mathring g_{jk})N\\
&\quad+nt\partial_k(N)-g^{ij}\onab_k(K_{ij}+t\mathring g_{ij})N\Big)\nu_{\mathring{g}}^kdS_{\mathring g}\\
&\quad-2t^{-1}\m_{ADM,\mathring g}(g,R)+2(n-1)nt^{-1}\int_{B_R}(N-1)dV_g\Big)\\
&=\lim_{R\to\infty}\Big(
-2(n-1)t^{-1}\int_{\p B_R}
 \partial_{k}(N)\nu_{\mathring{g}}^kdS_{\mathring g}-2t^{-1}\m_{ADM,\mathring g}(g,R)\\
&\quad
-2t^{-2}\int_{\p B_R}
 g^{ij}\left(\nabla_i(K_{jk}+t\mathring g_{jk})-\nabla_k(K_{ij}+t\mathring g_{ij})\right)\nu^k_{\mathring g}dS_{\mathring g}\\
&\quad +2(n-1)nt^{-1}\int_{B_R}(N-1)\Big)\\
&=\lim_{R\to\infty}\Big(
-2(n-1)t^{-1}\int_{\p B_R}
 \partial_{k}(N)\nu_{\mathring{g}}^kdS_{\mathring g}-2t^{-1}\m_{ADM,\mathring g}(g,R)\\
&\quad
-2t^{-1}\int_{\p B_R}
 g^{ij}\left(\nabla_i(\mathring g_{jk})-\nabla_k(\mathring g_{ij})\right)\nu^k_{\mathring g}dS_{\mathring g}\\
&\quad +2(n-1)nt^{-1}\int_{B_R}(N-1)dV_g\Big)\\
&=\lim_{R\to\infty}\Big(
-2(n-1)t^{-1}\int_{\p B_R}
 \partial_{k}(N)\nu_{\mathring{g}}^kdS_{\mathring g}-2t^{-1}\m_{ADM,\mathring g}(g,R)\\
&\quad
-2t^{-1}\int_{\p B_R}\mathring g^{ij}\left(\nabla_i(\mathring g_{jk}-g_{jk})-\nabla_k(\mathring g_{ij}-g_{ij})\right)\nu^k_{\mathring g}dS_{\mathring g}\\
&\quad +2(n-1)nt^{-1}\int_{B_R}(N-1)dV_g\Big)\\
&=\lim_{R\to\infty}2(n-1)\Big(
-t^{-1}\int_{\p B_R}
 \partial_{k}(N)\nu_{\mathring{g}}^kdS_{\mathring g}+nt^{-1}\int_{B_R}(N-1)dV_g\Big)
\end{split} \]
where we have made use of the fact that $(g-\mathring g),(K+t\mathring g)$, $N-1$, and $\partial N$ are all $O(r^{-\tau})$, as well as the momentum constraint. In order to simplify this further, we recall the elliptic equation for the lapse \eqref{eq-elliptic-N}, which can be expressed as
\[
\Delta N - |K_0|^2N = n(N-1),
\]
where $K_0$
is the (t-rescaled) traceless part of $K$ defined by
\[
K=t(K_0-g).
\]
Substituting this into the evolution equation for $\m_{VR,\mathring g}(g)$ and applying the divergence theorem, we find
\begin{equation}\label{eq-dm} \begin{split}
\frac{\partial }{\partial t}(\m_{VR,\mathring g}(g))
&=-2(n-1)t^{-1}\int_M\left|K_0\right|_g^2N dV_g.
\end{split} \end{equation}
That is, we see that the volume-renormalized mass is non-increasing along the evolution. Furthermore, $\frac{\partial }{\partial t}(\m_{VR,\mathring g}(g))=0$ if and only if $K_0$ vanishes. Suppose $\frac{\partial }{\partial t}(\m_{VR,\mathring g}(g))$ vanishes at some $t_0$ then $K_0=0$ at $t_0$ and from \eqref{eq-dm} it is clear that the second derivative of $\m_{VR,\mathring g}(g)$ also vanishes at $t_0$. Next, taking a 3rd derivative yields
\begin{equation*} \begin{split}
		\frac{\partial^3 }{\partial t^3}(\m_{VR,\mathring g}(g))
		&=-4(n-1)t^{-1}\int_M\left|\frac{\partial}{\partial t}K_0\right|_g^2N dV_g,
\end{split} \end{equation*}
which is strictly negative unless $\frac{\partial}{\partial t}K_0=0$ at $t_0$. Now from the elliptic equation for the lapse \ref{eq-elliptic-N} we conclude $N\equiv1$ and the evolution equations \ref{eq-evoeqs} then imply
\begin{equation*}
	\ric_{\widetilde\gamma}=-n(n-1)\widetilde\gamma.
\end{equation*}
That is, the development of the initial data determined by $(\widetilde\gamma,p)$ is Milne-like.
\end{proof}
\begin{rem}
	One may expect that the reduced Hamiltonian itself is monotone, rather than $t^{2-n}$ multiplied by it, as is the case in the compact setting \cite{FischerMoncrief2002}. This discrepancy arises due to the fact the relationship between $\tau$ and $t$ is fixed in our setting by the requirement that $N\to1$ at infinity, while in the compact setting the relationship between $\tau$ and $t$ can be prescribed. 
\end{rem}

%%%%%%%%%%%%%%%%%%%%%%%%%%%%%%%%%%%%%%%%%%%%%%%%%%%%%%%%%%%%%%%%%%%%%%%%%%%%%%%%
\bibliographystyle{abbrvurl}
\bibliography{AHmass}
%%%%%%%%%%%%%%%%%%%%%%%%%%%%%%%%%%%%%%%%%%%%%%%%%%%%%%%%%%%%%%%%%%%%%%%%%%%%%%%%

%%%%%%%%%%%%%%%%%%%%%%%%%%%%%%%%%%%%%%%%%%%%%%%%%%%%%%%%%%%%%%%%%%%%%%%%%%%%%%%%
\end{document}